\documentclass[12pt]{amsart}
\usepackage{amsfonts}
\usepackage{enumerate}
\usepackage{helvet,courier}
\usepackage{mathptmx,amsmath,amsxtra}
\usepackage{type1cm}
\DeclareMathAlphabet{\mathcal}{OMS}{cmsy}{m}{n}
\DeclareMathAlphabet{\mathbbold}{U}{bbold}{m}{n}  
\usepackage{amsthm}
\usepackage{graphicx}       
\usepackage{multicol}
\usepackage{mathrsfs}
\usepackage{amssymb}
\usepackage{verbatim}

\usepackage[usenames,dvipsnames]{xcolor}

\theoremstyle{plain}
\newtheorem{thm}{Theorem}
\newtheorem{lm}[thm]{Lemma}

\newtheorem{prop}[thm]{Proposition}

\theoremstyle{remark}
\newtheorem{rmk}{Remark}

\theoremstyle{definition}

\newcommand{\bnu}{\begin{enumerate}}
\newcommand{\enu}{\end{enumerate}}
\newcommand{\bpf}{\begin{proof}}
\newcommand{\epf}{\end{proof}}

\newcommand{\q}{\quad}
\newcommand{\qq}{\qquad}

\newcommand{\al}{\alpha}
\newcommand{\be}{\beta}
\newcommand{\ga}{\gamma}

\newcommand{\om}{\omega}

\newcommand{\ep}{\epsilon}

\newcommand{\si}{\sigma}
\newcommand{\tht}{\theta}

\newcommand{\vp}{\varphi}
\newcommand{\de}{\delta}

\newcommand{\bbr}{\mathbb{R}}

\newcommand{\rn}{\mathbb{R}^n}

\newcommand{\lp}{L^p}

\newcommand{\f}{\frac}

\newcommand{\p}{\partial}
\newcommand{\nf}{\infty}

\newcommand{\tf}{\tfrac}
\newcommand{\wh}{\widehat}

\newcounter{question}

\newcommand{\mm}{\mathcal M}
\newcommand{\wdt}{\widetilde}

\allowdisplaybreaks

\makeindex         

\begin{document}

\author[J. A.  Barrionevo]{J. A.  Barrionevo}

\address{Department of Pure and Applied Mathematics, Universidade Federal do Rio Grande do Sul 
Porto Alegre, RS, Brazil 91509-900}
\email{josea@mat.ufrgs.br}

\author[L. Grafakos]{Loukas Grafakos}
\address{Department of Mathematics, University of Missouri, Columbia MO 65211, USA}
\email{grafakosl@missouri.edu}

\author[D. He]{Danqing He}
\address{Department of Mathematics, Sun Yat-sen  University, Guangzhou, 510275, P. R. China}
\email{hedanqing@mail.sysu.edu.cn}

\author[P. Honz\'ik]{Petr Honz\'ik}
\address{MFF UK,
Sokolovska 83,
Praha 7,
Czech Republic}
\email{honzik@gmail.com}

\author[L. Oliveira]{Lucas Oliveira}

\address{Department of Pure and Applied Mathematics, Universidade Federal do Rio Grande do Sul 
Porto Alegre, RS, Brazil 91509-900}
\email{lucas.oliveira@ufrgs.br}

\title{Bilinear Spherical Maximal Function}
\date{}

\thanks{The fourth author was supported by the ERC CZ grant LL1203 of the Czech Ministry of Education. 
The second author acknowledges the support of Simons Foundation and of the University of Missouri Research Board and Research Council.}

\maketitle

\begin{abstract}
We obtain boundedness for the bilinear spherical maximal function in a range of exponents that includes 
the Banach triangle and a range of  $L^p$ with $p<1$.  We also obtain counterexamples that are asymptotically optimal 
with our positive results on  certain indices as the dimension tends to infinity.
\end{abstract}

\section{Introduction}
Let $\si$ be surface measure on the unit sphere.
The spherical maximal function
\begin{equation}\label{BSMFL}
\mathscr M(f)(x)=\sup_{t>0}\Big|\int_{|y|=1}f(x-ty)d\si(y)\Big|,
\end{equation}
was first studied by Stein \cite{Stein1976}  who
 provided a counterexample showing that it is unbounded on $\lp(\mathbb R^n)$ for 
$p\le \tf n{n-1}$  and   obtained the a priori  inequality
$\|\mathscr M(f)\|_{L^p(\mathbb R^n)}\le C_{p,n}\|f\|_{L^p(\mathbb R^n)}$ when $n\ge 3$, $p\in(\tf n{n-1},\nf)$ for 
smooth functions $f$; see also the account in \cite[Chapter XI]{stein-book2}. 
The extension of this result to the case $n=2$ 
was established about a decade  later by Bourgain    \cite{Bourgain1986}. 

In addition to Stein  and Bourgain, other authors have studied the spherical maximal function;  for instance see \cite{Cowling1979}, \cite{Carbery1985},
\cite{RdF86}, \cite{Mockenhaupt1992}, and \cite{Schlag1998}. 
Among   the techniques used in these works, we  highlight that of   Rubio de Francia 
\cite{RdF86}, in which  the $\lp$ boundedness of \eqref{BSMFL} is reduced to certain 
$L^2$ estimates obtained by Plancherel's theorem. Extensions of the spherical maximal function to different  settings have also been established by several authors:   for instance  see \cite{Coifman1978a}, \cite{Calderon1979} \cite{Greenleaf1981},
\cite{Duoandikoetxea1996} and \cite{Magyar2002}.

In this work we 
 study   the bi(sub)linear spherical maximal function  defined in \eqref{BSMF},
 which was   introduced and first  studied by \cite{GGIPS}.   In the bilinear setting the role of the crucial $L^2\to L^2 $ estimate is played  by an $L^2\times L^2\to L^1$,  
and obviously Plancherel's identity cannot be used on $L^1$. 
 We overcome the lack of orthogonality on $L^1$ via a wavelet technique introduced by three of the authors in
\cite{GraHeHon} 
in the study of certain bilinear operators; on this approach see   \cite{GHH1},    \cite{He}.
It is worth mentioning a related interesting recent  paper \cite{GIKL},
where the authors studied the bilinear circular average when $n=1$.
Our object of study here is the bi(sub)linear spherical maximal function 
\begin{equation}\label{BSMF}
\mm (f,g)(x)=\sup_{t>0}  \Big|\int_{\mathbb S^{2n-1}}f(x-ty)g(x-tz) d\si(y,z) \Big|
\end{equation}
initially defined for Schwartz functions $f,g$ on $\rn$. Here $\sigma$ is surface measure on the $2n-1$-dimensional sphere. We are concerned with bounds for $\mm$ from a product of 
Lebesgue spaces $L^{p_1}(\rn)\times L^{p_2}(\rn)$ to another Lebesgue space $L^p(\rn)$, where  
$1/p=1/p_1+1/p_2$. The main result of this article is the following:

\begin{thm}\label{Main}
Let $n\ge 8$ and let $\delta_n =(2n-15)/10$. Then the bilinear maximal operator $\mm$, 
 when restricted to Schwartz functions, is  bounded    from $L^{p_1}(\rn)
\times L^{p_2}(\rn)$ to $L^p(\rn)$ with $\f 1 p=\f 1{p_1}+\f 1 {p_2}$ for all indices $(\tf1{p_1},\tf1{p_2},\tf1p)$ in the  open  rhombus with vertices the points
$\vec P_0=(\f  1 \nf,\f  1 \nf,\f  1 \nf)$, $\vec P_1=(1,\f  1 \nf,  1)$, $\vec P_2=(\f 1 \nf, 1,1)$ and 
$\vec P_3=(\f{1+2\de_n}{2+2\de_n}, \f{1+2\de_n}{2+2\de_n},\f{1+2\de_n}{1+ \de_n})$.  
\end{thm}

Once Theorem~\ref{Main} is known, it follows that 
 $\mm$ admits a bounded extension from $L^{p_1}(\rn)
\times L^{p_2}(\rn)$ to $L^p(\rn)$ for   indices in the open rhombus of Theorem~\ref{Main} (for such indices we have 
$p_1,p_2<\nf$).  Indeed, 
given   $\{f_j\}_j$  Schwartz functions converging   to $f$ in $ L^{p_1}$ and   $\{g_k\}_k$ Schwartz functions converging      to $g$  in $L^{p_2}$,   we have that 
$$
\| \mm(f_j,g_j) - \mm(f_{j'}, g_{j'}) \|_{L^p} \le \big\|   \mm(f_j-f_{j'} ,g_j)  + \mm(f_j,g_j-g_{j'})   \big\|_{L^p} .
$$
It follows  from this  that the sequence $\{\mm(f_j,g_j)\}_j$ is Cauchy in $L^p(\rn)$ and hence it converges to a 
value which we also call 
$  \mm(f,g)$. This is the bounded extension of $\mm$ from $L^{p_1}(\rn)
\times L^{p_2}(\rn)$ to $L^p(\rn)$.  In order to pass to the maximal function defined on $L^{p_1}\times L^{p_2},$ it is also possible to used the technique desribed in~\cite[page 508]{stein-book2}.

Concerning  dimensions smaller than $8$, we have  positive answers in the Banach range
in next section. 

\section{The Banach range in dimensions $n\ge 2$}
 
\begin{prop}\label{infty} { Let $n\ge 2$. Then}
 $\mm$ maps $L^{p_1}(\rn) \times L^{p_2}(\rn) $ to $L^p (\rn)$  when $\f 1{p_1} +\f{1}{p_2} = \f 1p$, $1<p_1,p_2\le \nf$, and 
 $1<p\le \nf$. 
\end{prop}
 
\begin{proof}
We show that    $\mm$ is bounded on the intervals
$[\vec P_0,\vec P_1)$ and $[\vec P_0,\vec P_2)$, where $\vec P_1$ and $\vec P_2$ are as in Theorem~\ref{Main}.  
Then the claimed assertion follows by interpolation. 
If one function, for instance the second one $g$, lies in $L^\nf$, matters reduce to the $L^p(\rn)$ boundedness of the 
maximal operator
$$
\mm^0 (f )(x)=\sup_{t>0}  \int_{\mathbb S^{2n-1}}|f(x-ty) | d\si(y,z), 
$$
since $\mm(f,g)(x)\le \|g\|_{L^\nf}\mm^0(f)(x)$.
This expression inside the supremum is a Fourier multiplier operator of the form
$$
\int_{\bbr^{2n}}\wh {|f|}(\xi)\de_{0}(\eta)\wh{d\si}(t\xi,t\eta)e^{2\pi ix\cdot(\xi+\eta)}d\xi d\eta
=\int_{\bbr^{n}}\wh {|f|}(\xi)\wh{d\si}(t\xi,0)e^{2\pi ix\cdot\xi}d\xi 
$$
where $\de_0$ is the Dirac mass and 
$$
\wh{d\si}(t(\xi,0)) =  2\pi \,   \frac{J_{n-1} (2 \pi t |(\xi,0)| )}{|t (\xi,0) |^{n-1}}.
$$ 
The multiplier $\wh{d\si} (\xi,0) $ is smooth everywhere and decays like $|\xi|^{-(n-\frac12)}$ as $|\xi|\to \nf$ and 
its gradient has a similar decay.

The following result is   in \cite[Theorem B]{RdF86} (see also \cite{DRdF}): 

\noindent {\bf Theorem A.}
{\it 
Let $m(\xi)$ be a $\mathcal C^{[n/2]+1}(\rn)$ function that satisfies $|\p^\ga m(\xi)|\le (1+|\xi|)^{-a}$  for all $|\ga|\le [n/2]+1$ with $a\ge (n+1)/2$. Then
the maximal operator 
$$
f\mapsto \sup_{t>0} \big| \big( \wh{f}(\xi) m(t\xi) \big) \spcheck \big| 
$$
maps $L^p(\rn)$ to itself for $1<p<\nf$. 
}

{ In order to have $n-\f12\ge \f{n+1}{2}$ we must assume that $n\ge 2$. } 
It follows from Theorem A that $\mm^0$ is bounded on $L^p$ when $1<p\le \nf$ and $n\ge 2$. This completes the proof of Proposition \ref{infty}.
 
\end{proof}

\section{The point $(2,2,1)$}
Next we turn to the main estimate of this article which concerns the point $L^2\times L^2\to L^1$,  
i.e., the estimate $\|\mm(f,g)\|_{L^1}\le \|f\|_{L^2}\|g\|_{L^2}$. 

\begin{prop}\label{small} 
If $\psi$ is  in $\mathcal C_0^{\nf}(\mathbb R^{2n})$, then the maximal function
$$
M(f,g)(x)=\sup_{t>0}\bigg|\int_{\mathbb R^{2n}}\wh f(\xi)\wh g(\eta)
\psi(t\xi,t\eta)e^{2\pi ix\cdot(\xi+\eta)}d\xi d\eta\bigg|
$$
satisfies that for any $1<p_1,p_2<\nf$ and $1/p=1/p_1+1/p_2$, there exists
a constant $C$ independent $f$ and $g$ such that
$$
\|M(f,g)\|_{L^p(\mathbb R^n)}\le C\, \|f\|_{L^{p_1}(\mathbb R^n)}
\|g\|_{L^{p_2}(\mathbb R^n)}
$$
\end{prop}

The proof of Proposition~\ref{small} is standard and is omitted. 
Next, we decompose $\mm$. We fix $\vp_0\in\mathcal C_0^{\nf}(\mathbb R^{2n})$ such that
$ 
\chi_{B(0,1)}\le \vp_0\le\chi_{B(0,2)}
$ 
and we let $\vp(\xi,\eta) =\vp_0 ((\xi,\eta)) -\vp_0(2(\xi,\eta))$. 
For $j\ge 1$ define
$$m_j(\xi,\eta)= \wh{d\si}(\xi,\eta) \vp(2^{-j}(\xi,\eta)) $$ and for $j=0$ define 
$m_0(\xi,\eta)= \wh{d\si}(\xi,\eta) \vp_0 (\xi,\eta)  $. Then we have 
$$
\wh{d\si}=m=\sum_{j\ge0}m_j
$$
where  $\wh{d\si}(\xi,\eta) = 2\pi \frac{ J_{n-1} (2\pi (\xi,\eta))}{|(\xi,\eta)|^{n-1}}. $ 
Setting 
$$
\mm_j (f,g)(x)=\sup_{t>0}\bigg|\int_{\mathbb R^{2n}}\wh f(\xi)\wh g(\eta)
m_j(t\xi,t\eta)e^{2\pi ix\cdot(\xi+\eta)}d\xi d\eta\bigg| ,
$$
  we have the pointwise estimate 
\begin{equation}\label{sum}
\mm(f,g)(x)\le \sum_{j\ge0}\mm_j(f,g)(x) , \qq x\in \mathbb R^n.
\end{equation}

\begin{prop}\label{De}
For $n\ge 8$,
there exist positive constants $C$ and $\de_n =\f n5-\f32$ such that for all   $j\ge1$  and all functions $f,g\in L^2(\mathbb R^n)$ we have 
\begin{equation}\label{Deq}
\|\mm_j(f,g)\|_{L^1}\le C \, j \, 2^{-\de_n j}\|f\|_{L^2}\|g\|_{L^2}.
\end{equation}
\end{prop}

  Proposition~\ref{De} will be proved in the next section. In the remaining of this section we state and prove a lemma needed for its proof. 

\begin{lm}\label{Dia}
Suppose that $\si_1(\xi,\eta)$ is defined on $\mathbb R^{2n}$ and for some $\de>0$ it satisfies:  

(i) for any multiindex  $|\al|\le M=4n$, there exists a positive constant $C_{\al}$ independent of $j$ such that $\|\p^\al(\si_1(\xi,\eta))\|_{L^{\nf}}\le C_{\al}2^{-j\de}$,

(ii) supp $\si_1\subset \{(\xi,\eta)\in\mathbb R^{2n}: |(\xi,\eta)|\sim 2^j, c_12^{-j}\le 
\f{|\xi|}{|\eta |}\le c_22^j\}$.\\
Then $T(f,g)(x):=\int_0^{\nf}|T_{\si_t}(f,g)(x)|\f{dt}t$ is bounded from
$L^2(\rn)\times L^2(\rn)$ to $L^1(\rn)$ with   bound at most  a multiple of $j\|\si_1\|_{L^2}^{4/5}2^{-j\de/5}$,
where $\si_t(\xi,\eta)=\si_1(t\xi,t\eta)$.
\end{lm}

\begin{proof}[Proof of Lemma~\ref{Dia}]
A crucial tool in the proof of Lemma \ref{Dia} is the following result  \cite[Corollary 8]{GraHeHon}:  

\medskip
\noindent{\bf Proposition B.} {\it 
Let $m \in L^2(\mathbb R^{2n})$ and $C_M>0$ satisfy 
 $\|\p^\al m\|_{L^{\nf}}\le C_{M}$ for each multiindex  $|\al |\le M=16n$. 
Then the bilinear operator $T_m$ associated with  the multiplier $m$ satisfies 
 $$
 \|T_m \|_{L^2\times L^2\to L^1}\le C\, C_M^{1/5}\|m\|_{L^2}^{4/5} .
 $$
}

Using Proposition B, setting 
  $\wh {f^j}= \widehat{f}\chi_{\{c_1\leq |\xi| \leq c_2 2^{j+1}\}}$,  
by the support of $\si_1$ we obtain  that 
$$
\|T_{\si_1}(f,g)\|_{L^1}\le C\|\si_1\|_{L^2}^{4/5}2^{-j\de/5}
\|f^j\|_{L^2} \|g^j\|_{L^2} \, .
$$
 Notice that
$T_{\si_t}(f,g)(x)=t^{-2n}T_{\si_1}(f_t,g_t)(\f xt)$, where 
$\wh{f_t}(\xi)=\wh f(\xi/t)$.
Then 
\begin{align*}
\|T_{\si_t}(f,g)\|_{L^1}\,\,\le &\,\, C\|\si_1\|_{L^2}^{4/5}2^{-j\de/5}t^{-n}\|\wh f(\xi/t)\chi_{E_{j,0}}\|_{L^2}\|\wh g(\eta/t)\chi_{E_{j,0}}\|_{L^2}\\
= &\,\, C\|\si_1\|_{L^2}^{4/5}2^{-j\de/5} \|\wh f\,\chi_{E_{j,t}}\|_{L^2}\|\wh g\,\chi_{E_{j,t}}\|_{L^2},
\end{align*}
where $E_{j,t}=\{\xi\in\mathbb R^n: \f{c_1}t\leq |\xi| \leq \f{2^{j}c_2}t\}$.

As a result we obtain 
\begin{align*}
&\int_{\mathbb R^n}\int_0^{\nf}|T_{\si_t}(f,g)|\f{dt}tdx\\
\le&\, C\|\si_1\|_{L^2}^{4/5}
2^{-j\de/5}\int_0^{\nf}\|\wh f\, \chi_{E_{j,t}}\|_{L^2}\|\wh g\,\chi_{E_{j,t}}\|_{L^2}\f{dt}t\\
\le &\, C\|\si_1\|_{L^2}^{4/5}
2^{-j\de/5}\bigg(\int_0^{\nf}\int_{\mathbb R^n}|\wh f\,\chi_{E_{j,t}}|^2d\xi\f{dt}{t}\bigg)^{\frac12}
\bigg(\int_0^{\nf}\int_{\mathbb R^n}|\wh g\, \chi_{E_{j,t}}|^2d\xi\f{dt}{t}\bigg)^{\frac12}.
\end{align*}

We control the last term as follows:
$$
\int_0^{\nf}\int_{\mathbb R^n}|\wh f\, \chi_{E_{j,t}}|^2d\xi\f{dt}{t}\le 
C\int_{\mathbb R^n}\int_{1/|\xi|}^{2^j/|\xi|}\f{dt}{t}|\wh f(\xi)|^2d\xi\le Cj\|f\|_{L^2}^2
$$
and thus we deduce  
$$
\|T(f,g)(x)\|_{L^1}\le C\|\si_1\|_{L^2}^{4/5}2^{-j\de/5}j\|f\|_{L^2}\|g\|_{L^2}   .
$$ 
This completes the proof of Lemma~\ref{Dia}.
\end{proof}

\section{Proof of Proposition~\ref{De}}

\begin{proof}
Estimate \eqref{Deq} is automatically holds for finitely many terms in view of Proposition \ref{small},
so we fix a large $j$ and define
$$
T_{j,t}(f,g)(x)=\int_{\mathbb R^{2n}}\wh f(\xi)\wh g(\eta)
m_j(t\xi,t\eta)e^{2\pi ix\cdot(\xi+\eta)}d\xi d\eta .
$$

Take a smooth function $\rho$ on $\mathbb R$ such that $\chi_{[\ep-1,1-\ep]}\le\rho\le \chi_{[-1,1]}$.
Define $m_j^1(\xi,\eta)=m_j(\xi,\eta)\rho(\tf1j(\log_2\tf{|\xi|}{|\eta|}))$, then we have a smooth
  decomposition of $m_j$ with 
$m_j=m_j^1+m_j^2$.
On the support of $m_j^{1}$ we have $C^{-1}2^{-j}|\xi| \le |\eta|\le C2^j|\xi|$ and on the support of 
$m_j^{2}$ we have  $2^{j(1-\ep)}|\xi| \lesssim |\eta|$ or $2^{j(1-\ep)}|\eta| \lesssim |\xi|$. 
We define
$$
\mm^i_j(f,g)=\sup_{t>0}|T_{j,t}^i(f,g)|,\q i\in \{1,2\},
$$
where 
$T_{j,t}^1$ and $T_{j,t}^2$ correspond to multipliers $m_j^1(t(\xi,\eta))$
and $m_j^2(t(\xi,\eta))$ respectively, such that
$T_{j,t}=T_{j,t}^1+T_{j,t}^2$. Then for $f,g$ Schwartz functions we have
\begin{align*}
  \mm_j^1(f,g)(x)=&\sup_{t>0}|T^1_{j,t}(f,g)(x)|\\
=&\sup_{t>0} \bigg|\int_0^ts\f{dT_{j,s}^1(f,g)}{ds}\f{ds}s\bigg|\\
\le & \int_0^{\nf}|\widetilde{T}_{j,s}^1(f,g)(x)|\f{ds}s,
\end{align*}
where $\widetilde{T}_{j,s}^1$ has  bilinear  multiplier
$\widetilde{m}^1_{j}(s\xi,s\eta)=(s\xi,s\eta)\cdot (\nabla m^1_j)(s\xi, s\eta)$,  a diagonal multiplier with nice
decay, which can be used to  establish the boundedness of the diagonal 
part  with the aid of Lemma~\ref{Dia}.

Recall that 
$$
m_j^1(\xi,\eta)=  \vp (2^{-j}( \xi,\eta)) 2\pi \frac{ J_{n-1} (2\pi (\xi,\eta))}{|(\xi,\eta)|^{n-1}}\rho(\tf1j(\log_2\tf{|\xi|}{|\eta|}))
$$
 for $j\ge 1$ 
and a calculation shows that $|\p_1(m_j^1)|$ is controlled by the sum of three terms      bounded by 
$C2^{-j(2n-1)/2}$, $C2^{-j(2n+1)/2}$ and $C\f1j2^{-j(2n-1)/2}$ respectively.
Indeed, when the derivative falls on $\phi$, we can bound it by 
$C2^{-j}2^{-j(n-1/2)}$ $= C2^{-j(n+1/2)}$.
If the derivative falls on the second part, using properties of Bessel functions (see, e.g., \cite[Appendix B.2]{CFA}),
we obtain the bound $C\frac{ J_{n} (2\pi (\xi,\eta))}{|(\xi,\eta)|^{n}}|\xi_1|\le C2^{-j(n- 1/2)}$.
For the last case, we can bound it by $C2^{-j(n- 1/2)}j^{-1}\f1{|\xi|}\f{\xi_1}{|\xi|}\le C 2^{-j(n- 1/2)}j^{-1}2^{-\ep j}$.
As a consequence we have  
$|\p_1(m_j^1)|\le C2^{-j(2n-1)/2}$. Then we can show that 
$|\p_1(\wdt m_j^1)|\le C2^{-j(2n-3)/2}$ and similar arguments give that
for any multiindex $\al$ we have $|\p^{\al}\wdt m_j^1|\le C2^{-j(2n-3)/2}$. 
Moreover, from this we can show  that
$$
\|\wdt m_j^1\|_2
\le C \bigg(\int_{|(\xi,\eta)|\sim 2^j}|2^{-j(n-\f32)}|^2d\xi d\eta\bigg)^{\f12}
\le C2^{-j(n-\f32)}2^{jn}\le C2^{\f{3 }2j} .
$$

 Applying Lemma \ref{Dia} 
to the function  $\wdt m_j^1(\xi,\eta)=(\xi,\eta)\cdot (\nabla m^1_j)(\xi, \eta)$ which satisfies the hypotheses with $\de=(2n-3)/2$, 
we obtain 
\begin{equation}\label{E1}
\|\mm_j^1(f,g)\|_{L^1}\le Cj \|\wdt m_j^1\|_{L^2}^{\f45} \, 2^{-j \f\de{5}}\|f\|_{L^2}\|g\|_{L^2}  
= Cj2^{j(\f32-\f n5)}\|f\|_{L^2}\|g\|_{L^2}.
\end{equation}
It remains to obtain an analogous estimate for $\mm_j^2$.

For the off-diagonal part $m_j^2$     we use a different decomposition involving 
$g$-functions. For $f,g\in \mathcal S(\mathbb R^n)$ we have 
\begin{align}
\mm_j^2(f,g)(x)\,=&\, \big(\sup_{t>0}|T^2_{j,t}(f,g)(x)|^2\big)^{\frac12}\notag \\
=&\,\bigg(\sup_{t>0}\bigg|2\int_0^tT^2_{j,s}(f,g)(x) \, s\f{dT_{j,s}^2(f,g)(x)}{ds}\f{ds}s\bigg|\bigg)^{\frac12}\notag \\
\le &\,\sqrt{2} \, \Big\{\Big(\int_0^{\nf}|T_{j,s}^2(f,g)|^2\f{ds}s\Big)^{\frac12}\Big(\int_0^{\nf}|\widetilde{T}_{j,s}^2(f,g)|^2\f{ds}s\Big)^{\frac12}\Big\}^{\frac12} \notag \\
=&\,\sqrt{2} \, \big(G_j(f,g)(x)\widetilde G_j(f,g)\big)^{\frac12}. \label{mmid}
\end{align}
Here   $\widetilde{T}_{j,s}^2(f,g)$ has symbol      
$ 
\wdt m^2_j(s\xi,s\eta)=(s\xi,s\eta)\cdot(\nabla m_j^2)(s\xi,s\eta)
$ 
 and 
\begin{eqnarray*} 
G_j(f,g)(x)&\! \! = \! \! & \Big(\int_0^{\nf}|T_{j,s}^2(f,g)|^2\f{ds}s\Big)^{\frac12} \\
 \wdt G_j(f,g)(x)&\! \!  =\! \!  &\Big(\int_0^{\nf}|\wdt T_{j,s}^2(f,g)|^2\f{ds}s\Big)^{\frac12}.
 \end{eqnarray*}

\begin{lm}\label{ODia}
If a $\si_1(\xi,\eta)$   on $\mathbb R^{2n}$ satisfies 

(i) for any multiindex  $|\al|\le M=4n$, there exists a positive constant $C_{\al}$ independent of $j$ such that $\|\p^\al(\si_1(\xi,\eta))\|_{L^{\nf}}\le C_{\al}2^{-j\de}$,

(ii) supp $\si_1\subset \{(\xi,\eta)\in\mathbb R^{2n}: |(\xi,\eta)|\sim 2^j, |\xi|\ge 2^{j(1-\ep)}|\eta|, \text{ or } |\eta|\ge 2^{j(1-\ep)}|\xi|\}$,\\
then $T(f,g)(x):=(\int_0^{\nf}|T_{\si_t}(f,g)(x)|^2\f{dt}t)^{1/2}$ is bounded from
$L^2\times L^2$ to $L^1$ with   bound at most  a multiple of $2^{-j(\de-\ep)}$,
where $\si_t(\xi,\eta)=\si_1(t\xi,t\eta)$.
\end{lm}

\begin{proof}
 Recall that $\text{supp } m_j^{2}\subset \{(\xi,\eta): 2^{j(1-\ep)}|\xi| \lesssim |\eta|\text{ or }2^{j(1-\ep)}|\eta| \lesssim |\xi|\}$.
We consider only the part $\{|\xi|\ge  2^{j(1-\ep)}   |\eta|\}$ because the other part is similar. By
\cite[Section 5]{GraHeHon}  we have
$$
|T_{\si_1}(f,g)(x)|\le C2^{\ep j}2^{-j\de}M(g)(x)| T_m(f)(x)|,
$$
where
$M$ is the Hardy-Littlewood maximal function and  $T_m$ is a linear operator that satisfies 
$\|T_m(f)\|_{L^2}\le C\|\wh f\chi_{\{|\xi|\sim 2^j\}}\|_{L^2}$.
Then 
$$
|T_{\si_t}(f,g)(x)|\le 2^{-j(\de-\ep)}t^{-n}M(g)(x)T_m(f_t)(x/t) ,
$$
and 
\begin{align*}
\int_{\mathbb R^n}\bigg(\int_0^{\nf} &|T_{\si_t}(f,g)(x)|^2\f{dt}t\bigg)^{\frac12}dx\\
\le &\  C 2^{-j(\de-\ep)}\int_{\mathbb R^n}\bigg(\int_0^{\nf}t^{-2n}M(g)(x)^2
| T_m(f_t)(x/t) |^2 \frac{dt}{t} \bigg)^{\frac12}dx\\
\le &\   C2^{-j(\de-\ep)}\|M(g)\|_{L^2}\bigg (\int_{\mathbb R^n}
\int_{0}^{\nf} \big| t^{-n} T_m(f_t)(x/t) \big|^2 \frac{dt}{t} dx\bigg)^{\frac12}\\
\le &\   C2^{-j(\de-\ep)}\|g\|_{L^2}\Big(\int_{\mathbb R^n}|\wh f(\xi)|^2\int_{2^{j-1}/|\xi|}^{2^{j+1}/|\xi|}\f{dt}td\xi\Big)^{\frac12}\\
\le &\   C2^{-j(\de-\ep)}\|g\|_{L^2}\|f\|_{L^2}.
\end{align*}
This completes the proof of Lemma~\ref{ODia}.\end{proof}

We now return to the proof of Proposition \ref{De}. Notice that both $m^2_j(\xi,\eta)$ and $\wdt m^2_j(\xi,\eta)$
satisfy conditions of Lemma \ref{ODia} with $\de$ being either $(2n-1)/2$ or 
$(2n-3)/2$ respectively, so 
\begin{align*}
\|G_j(f,g)\|_{L^1}&\le \,   C2^{-j(2n-1)/2}\|f\|_{L^2}\|g\|_{L^2}\\ 
 \|\wdt G_j(f,g)\|_{L^1}&\le \,  C2^{-j(2n-3)/2}\|f\|_{L^2}\|g\|_{L^2} .
 \end{align*}
Using \eqref{mmid} we deduce 
\begin{equation}\label{E2}
\|\mm_j^2(f,g)\|_{L^1} \le \|G_j(f,g)\|_{L^1}^{1/2}
\|\widetilde G_j(f,g)\|_{L^1}^{1/2}\le C2^{-j(n-1)}\|f\|_{L^2}\|g\|_{L^2}.
\end{equation}
Combining \eqref{E1} and \eqref{E2} yields   Proposition \ref{De} with $\de_n=\f n5 -\f32$.
\end{proof}

\section{Interpolation}

By Proposition \ref{small} (for  term $j\le c_0$) and Proposition \ref{De} (for $j\ge c_0$),  for any $\de_n'<\de_n$, 
as a consequence of \eqref{sum} we obtain 
\begin{align*}
\|\mm(f,g)\|_{L^1}\le &
\sum_{j=0}^{\nf}C_{\de'_0}2^{-\de_n'j}\|f\|_{L^2}\|g\|_{L^2}
\le C_{\de'_0}\|f\|_{L^2}\|g\|_{L^2}.
\end{align*}
This establishes the  boundedness of $\mm$ from $L^2\times L^2$ to $L^1$ claimed in Theorem~\ref{Main} (recall 
$n\ge 8$). 
It remains to obtain estimates for other values of $p_1,p_2$. This is   achieved via bilinear interpolation.


Notice that when one index among $p_1$ and $p_2$ is equal to $1$, we have that 
$\mm_j$ maps $L^{p_1}\times L^{p_2}$ to  $L^{p,\nf}$  with norm  $\lesssim 2^j$. Indeed, this follows from the estimate 
$$
|\vp_j^{\vee}*(d\si)(y,z)|\le C_{N}2^j(1+|(y,z)|)^{-2{N}}\le C_{{N}}2^j(1+|y|)^{-{{N}}}(1+|z|)^{-{{N}}}
$$
which can be found, for instance, in     \cite[estimate (6.5.12)]{CFA}. Thus  we have 
$$\mm_j(f,g)(x)\le C2^jM(f)M(g)$$ where  
$M$ is  the Hardy-Littlewood maximal function.
 We pick two points 
\begin{align*}
 \vec Q_1&=\, (1/1,1/(1+\varepsilon), (2+\varepsilon)/(1+\varepsilon)) \\
  \vec Q_2&=\,(1/(1+\varepsilon),1/1 , (2+\varepsilon)/(1+\varepsilon))
\end{align*}
  and we also consider the point 
$\vec Q_0=(1/2,1/2,1)$. We interpolate the known estimates for $\mm_j$ at these three points. Letting $\ep$ go to $0$,  we obtain
that for $p>\f{2+2\de_n}{1+2\de_n}$ we have that $\mm_j$ maps $L^{p}(\rn) \times L^{p}(\rn) $ to $L^{p/2}(\rn)$ 
with a geometrically decreasing bound   in $j$. 
Recall that $\de_n= (2n-15)/10>0 $, so we need $n\ge 8$. 

 Thus summing over $j$ gives boundedness for $\mm$ 
from $L^{p}(\rn) \times L^{p}(\rn) $ to $L^{p/2}(\rn)$ when $p>\f{2+2\de_n}{1+2\de_n}$.  By interpolation we obtain 
boundedness for $\mm$ in the interior of a rhombus with vertices the points
$(1/\nf,1/\nf,1/\nf)$, $(\f{2n-3/2}{2n-1}, \f{1}{\nf} ,  \f{2n-3/2}{2n-1})$, $(\f{1}{\nf}, \f{2n-3/2}{2n-1},  \f{2n-3/2}{2n-1})$ and 
$(\f{1+2\de_n}{2+2\de_n}, \f{1+2\de_n}{2+2\de_n},\f{2+4\de_n}{2+2\de_n})$. 
The proof of Theorem \ref{Main} is now complete.

We remark that  is the largest region for which we presently know boundedness for $\mm$ in dimensions $n\ge 8$. 


\section{Counterexmaples}
 In this section we construct  counterexamples   indicating the unboundedness of the bilinear spherical maximal operator 
in a certain range. Our examples are inspired by Stein \cite{Stein1976} but the situation  is more  complicated.

\begin{prop}
The bilinear spherical maximal operator $\mm$ is unbounded from $L^{p_1}(\bbr^n)\times L^{p_2}(\bbr^n)$
to $L^p(\bbr^n)$
when  $1\le p_1,p_2 \le \infty$, $\tf1p=\tf1{p_1}+\tf1{p_2}$, $n\ge 1$, and $p\le \tf n{2n-1}$. 
In particular, 
$\mm$ is unbounded from $L^2(\bbr)\times L^2(\bbr)$ to $L^1(\bbr)$ when $n=1$.
\end{prop}

\begin{rmk}
We note that $\tf{1+\de_n}{1+2\de_n}-\tf n{2n-1}= \f{1+\f n 5-\f32}{1+\f {2n}5-3} -\f{n}{2n-1}\approx \f1n\to 0$ as $n\to \infty$.
This means that the gap between the range of  boundedness and 
unboundedness tends to $0$ as the dimension increases to infinity.

\end{rmk}

\bpf
We first consider the case   $n=1$ where it is easy to demonstrate the main idea. 

Define functions on $\mathbb R $ by setting
 $f(y) =|y|^{-1/p_1}(\log \tf1{|y|})^{-2/p_1}\chi_{|y|\le 1/2}$  and $g(y)=|y|^{-1/p_2}(\log \tf1{|y|})^{-2/p_2}\chi_{|y|\le 1/2}$.
Then $f\in L^{p_1}(\mathbb R)$, $g\in L^{p_2}(\mathbb R)$ and we will estimate from below 
$M_{\sqrt 2R}(f,g)(R) $ for large $R$, where 
$$
M_t(f,g)(x)=  \int_{\mathbb S^{ 1}}|f(x-ty)g(x-tz) |d\si(y,z)  .
$$

  In view o  the support   properties  of $f$ and $g$  we have
$|y-\tf1{\sqrt 2}|\le\tf1{2\sqrt 2R}$,
and $|z-\tf1{\sqrt 2}|\le\tf1{2\sqrt 2R}$. We   also have 
that $y^2+z^2=1$ since $(y,z)\in\mathbb S^1$.

Therefore we   rewrite $M_{\sqrt 2R}(f,g)(R)$
as 
\begin{align}\begin{split}\label{kkklll}
\int_{\tf{\sqrt 2}2-\tf1{2\sqrt 2R}}^{\tf{\sqrt 2}2+\tf1{2\sqrt 2R}} 
&|R(1-\sqrt 2y)|^{-\tf1{p_1}}(-\log|R(1-\sqrt 2y)|)^{-\tf{2}{p_1}} \\
&|R(1-\sqrt 2z)|^{-\tf1{p_2}}(-\log|R(1-\sqrt 2z)|)^{-\tf{2}{p_2}}\tf{dy}{\sqrt{1-y^2}},
\end{split}
\end{align}
with $z=\sqrt{1-y^2}$.

Notice that $|R(1-\sqrt 2z)|=R|\tf{1-2z^2}{1+\sqrt 2z}|\le R{|1-2y^2|}
\le3R|1-\sqrt 2y|$
since\footnote{Here $a\approx b$ means that $|a-b|$ is very small.} $z\approx y\approx\sqrt 2/2$.
As a result, with the help of \eqref{e04021} [Lemma~\ref{lemmaO}], the expression in \eqref{kkklll} is greater than 
\begin{align*}
\int_{\tf{\sqrt 2}2-\tf1{100R}}^{\tf{\sqrt 2}2+\tf1{100R}} &
R^{-\tf 1p}
|(1-\sqrt 2y)|^{-\f 1p}(-\log|R(1-\sqrt 2y)|)^{-\f 2p }dy \\
= &
 2 R^{-1}\int_0^{\tf1{100} }t^{-1/p}(\log\tf 1t)^{-2/p}dt = \begin{cases}  C_p R^{-1} \q &\textup{if $p\ge 1$}\\  \infty &\textup{if $p< 1$.} \end{cases}
\end{align*}
Thus $\mathcal M(f,g)\notin L^p(\mathbb R)$ for $p<1  $ and also 
$\mathcal M(f,g)(x) \ge C /x$ for $x$ large if $p=1$. It follows that $\mathcal M(f,g)\notin L^1(\mathbb R)$ for $p=1$, 
hence   the statement of the proposition holds.

We now consider the higher-dimensional case $n\ge 2$.
We define $f(y)=|y|^{{ -n/p_1} }(\log \tf1{|y|})^{{ -2/p_1}}\chi_{|y|\le 1/100}$  
and $g(y)=|y|^{{ -n/p_2} }(\log \tf1{|y|})^{{ -2/p_2} }\chi_{|y|\le 1/2} $. 
{ We have that 
$f$ lies in $L^{p_1}(\mathbb R^n)$ and $g$ lies in $L^{p_2}(\mathbb R^n)$.}
The mapping $(y,z)\mapsto (Ay,Az)$ with $A\in SO_n$ is an isometry on $\mathbb S^{2n-1}$, hence 
{  we have $M_t(f,g)(x) = M_t(f,g)(|x|e_1)$, where $e_1=(1,0,\dots,0)\in \bbr^n $.
Thus we may take $x=Re_1\in \bbr^n$ with $R$ large.}

By the change of variables identity~\eqref{e04022} [Lemma \ref{04021}], we have 
\begin{align*}
M_{\sqrt 2R}&(f,g)(Re_1)  \\
=&
\int_{\mathbb S^{2n-1}} f(Re_1-\sqrt 2Ry) g(Re_1-\sqrt 2Rz)d\si(y,z)\\
=&\int_{B_{n}(\tf1{\sqrt 2}e_1,\tf 1{100R})}|\sqrt R y-Re_1|^{{ -\frac{n}{p_1}}}
(-\log|Re_1-\sqrt 2Ry|)^{ { - \f{2}{p_1}}}\\
&\quad\int_E
|\sqrt 2Rz-Re_1|^{{ -\frac{n}{p_2}}}(-\log|Re_1-\sqrt 2Rz|)^{{ -\frac{2}{p_2}}}d\si^r_{n-1}(z)\tf{dy}{\sqrt{1-|y|^2}},
\end{align*}
where $B_n(a,r)$ is a ball in $\rn$ centered at $a$ with radius $r$, 
and
$E$ is the $(n-1)$-dimensional manifold ${\mathbb S^{n-1}_{\sqrt{1-|y|^2}}\cap B_n(\tf1{\sqrt 2}e_1,\tf1{2\sqrt 2R})}$ with 
$\mathbb S^{n-1}_r$ being  the sphere in $\rn$ with radius $r$ and
 $d\si_{n-1}^{{  r}}$ the measure on $\mathbb S^{n-1}_r$.
 
We next focus on the inner integral, namely
$$
I=\int_E
|\sqrt 2Rz-Re_1|^{{ -\frac{n}{p_2}}}(-\log|Re_1-\sqrt 2Rz|)^{{ -\frac{2}{p_2}}}d\si_{n-1}^{{  r}}(z).
$$
Take a point $z_0\in \mathbb S^{n-1}_{\sqrt{1-|y|^2}}\cap {  \partial \big( B_n} (\tf1{\sqrt 2}e_1,\tf1{2\sqrt 2R})\big) $,
and let $\tht$ be the angle between vectors $ z_0$ and $e_1$,
which the largest one between $z\in E$ and $e_1$. {  Here $\partial B$ is the boundary of a set $B$.}
Then {  $\theta$ is small if $R$ is large} and\footnote{$A\sim B$ means that the ratio $A/B$ is bounded
above and below} $|E|\sim (\sqrt{1-|y|^2}\tht)^{n-1}\sim\tht^{n-1}$.
Noticing that $\tht^2\sim\sin^2\tht=1-\cos^2\tht\sim1-\cos\tht$ and   that 
$$1-|y|^2+\tf12-\sqrt2\sqrt{1-|y|^2}\cos\tht=\tf1{8R^2} ,$$  we obtain that
$\tht^2\sim\tf1{8R^2}-(\sqrt{1-|y|^2}-\tf1{\sqrt 2})^2$.
Then we write 
$$
\Big|\sqrt{1-|y|^2}-\tf1{\sqrt 2}\Big|=
\Big|\tf{1-|y|^2-\tf1{ 2}}{\sqrt{1-|y|^2}+\tf1{\sqrt 2}}\Big|
\le 2|\tf12-|y|^2|\le\tf1{25R}.
$$
Consequently $\tht\ge C/R$.

Collecting the previous calculations, we can bound $I$ from below by
$$
\int_0^\tht\int_{\mathbb S^{n-2}_{t\sin\al}}|\sqrt 2Rz-Re_1|^{{ -\frac{n}{p_2}}}(-\log|Re_1-\sqrt 2Rz|)^{{ -\frac{2}{p_2}}}
d\si_{n-2}^{{  t\sin \alpha}} (z)d\al,
$$
where $t=|z|=\sqrt{1-|y|^2}\approx\tf1{\sqrt2}$, and $z_1=\cos\al$.
By symmetry, let us consider just that case $t<\tf1{\sqrt 2}$.
Let $\be$ be the angle such that $|\sqrt 2z-e_1|=2|\sqrt 2t-1|$,
then $2t^2+1-2\sqrt 2t\cos\be=4|\sqrt 2t-1|^2$,
which implies that
$\be^2\sim1-\cos\be
\sim2\sqrt 2t-2t^2-1+4(\sqrt2t-1)^2=3(\sqrt 2t-1)^2.$
So $\be\sim 1-\sqrt 2t$. 
When $\al=0$, we have trivially that $|\sqrt2z-e_1|=|\sqrt 2t-1|$.
So for $\al\in[0,\be]$, we have $|\sqrt2z-e_1|\sim 2|\sqrt 2t-1|
\le 2\big|2|z|^2-1\big|=2\big|2|y|^2-1\big|\le 6\big|\sqrt 2|y|-1\big|\le 6|\sqrt 2y-e_1|$.
Consequently using { the fact that $1-\sqrt 2t\le C\tht$ and} \eqref{e04021} again we obtain
\begin{align*}
I\ge \ &C\int_0^\tht\int_{\mathbb S^{n-2}_{t\sin\al}} 
  \frac{ |\sqrt 2Rz-Re_1|^{1-n} } {|\sqrt 2Rz-Re_1|^{ {  \f{n}{p_2} -n+1 } }(-\log|Re_1-\sqrt 2Rz|)^{{ \f 2 {p_2}}  }}
\,\, d\si_{n-2}^{{  t\sin \alpha}}(z)\, d\al\\
\ge \ &  \f{CR^{1-n}     |\sqrt 2t-1|^{1-n} }{ |\sqrt 2Ry-Re_1|^{  {  \f{n}{p_2} -n+1 }  }(-\log|Re_1-\sqrt 2Ry|)^{ {  \f{2}{p_2}   }}  } \int_0^{C(1-\sqrt2t)}\sin^{n-2}\al d\al\\
\ge\ & CR^{1-n}   \f{|\sqrt 2t-1|^{1-n}|1-\sqrt2t|^{n-1}}{|\sqrt 2Ry-Re_1|^{   {  \f{n}{p_2} -n+1 }  }(-\log|Re_1-\sqrt 2Ry|)^{ {  \f{2}{p_2}   }}    }    
\\
=\ &CR^{1-n}|\sqrt 2Ry-Re_1|^{  {  -\f{n}{p_2} +n-1 }   }(-\log|Re_1-\sqrt 2Ry|)^{-{  \f{2}{p_2} } }.
\end{align*}
Using this estimate 
we see that
\begin{align*}
M_{\sqrt 2R}&(f,g)(Re_1) \\
\ge\ &C R^{1-n}\int_{B_n(\tf1{\sqrt2}e_1,\tf1{100R})}
|Re_1-\sqrt 2Ry|^{   { -\f np +  n-1}   }(-\log |Re_1-\sqrt 2Ry|)^{    { -\frac 2 p   }    }dy\\
=\ &CR^{   {   1-2n} }\int_{B_n(0,\tf1{100})}
|x|^{ { -\f np +  n-1}   }(-\log |x|)^{  { -\frac 2 p   } }dx\\
=\ &CR^{   {   1-2n} }\int_0^{\tf1{100}}r^{{ -\f np + 2n-2} }(-\log r)^{{ -\frac 2 p   } } dr\\
=\ &  {  \begin{cases}  CR^{ -2n+1}    &\q\textup{ if $p=\f n{2n-1}$}  \\ \infty   &\q\textup{ if $p<\f n{2n-1}$}\end{cases}  } .  
\end{align*}
Hence $\mm(f,g) $ is not in $L^p$ for $p<\f n{2n-1}$ and $\mm(f,g)(x)\ge C|x|^{1-2n}$ for all
$|x|$ large enough, hence it is also not in $L^{\f n{2n-1}}(\rn)$ when $p=\f n{2n-1}$.
\epf

Lastly, we prove a couple of points left open. 
 
\begin{lm}\label{lemmaO}
Let $r_1,r_2>0$, $t,\ s\le \tf1{10}$, and $t\le Cs$ for some $C\ge 1$.
Then there exists an absolute constant $C'$ (depending on $C, {  r_1,r_2}$) such that
\begin{equation}\label{e04021}
s^{-r_1}(\log\tf1s)^{-r_2}\le C' t^{-r_1}(\log\tf1t)^{-r_2}.
\end{equation}

\end{lm}

\bpf
Define $F(x)=x^{r_1}(\log x)^{-r_2}$. Differentiating $F$, we see that
$F$ is increasing when $x$ is large enough and so, 
$$
F(\tf1{s})=s^{-r_1}(\log\tf1s)^{-r_2}\le C^{r_1} (Cs)^{-r_1}(\log\tf1{Cs})^{-r_2}
=C^{r_1}F(\tf1{Cs})\le {  C'} F(\tf1t) ,
$$
which is a restatement of \eqref{e04021}.

\epf

\begin{lm}\label{04021}
For   functions $F(y,z)$ defined in $\bbr^{2n}$ with $y,\ z\in \rn$,
we have
\begin{equation}\label{e04022}
\int_{\mathbb S^{2n-1}}F(y,z)d\si(y,z)
= \int_{B_n}\int_{\mathbb S^{n-1}_{r_y}}F(y,z)d\si_{n-1}^{{  r_y}} (z)\tf{dy}{\sqrt{1-|y|^2}},
\end{equation}
where $B_n$ is the unit ball in $\rn$ and $\mathbb S^{n-1}_{r_y}$ is the sphere in $\rn$
centered at $0$ with radius $r_y=\sqrt{1-|y|^2}$.
\end{lm}

\bpf
We begin by writing   
$\int_{\mathbb S^{2n-1}}F(y,z)d\si(y,z)$ as 
\begin{equation}\label{e04023}
\int_{B_{2n-1}}{   \big[ F(y,z',z_n) +  F(y,z',-z_n) \big]  } \tf{dydz'}{\sqrt{1-|y|^2-|z'|^2}},
\end{equation}
where $z=(z',z_n)$, and $z_n=\sqrt{1-|y|^2-|z'|^2}$; see \cite[Appendix D.5]{CFA}. 

{  Writing $z/r_y=\om=(\om',\om_{n-1}) \in \mathbb R^{n-1}\times \mathbb R$, }we express  the right hand side of \eqref{e04022} as
\begin{align*}
&\int_{B_n}\int_{\mathbb S_{r_y}^{n-1}}F(y,z)d\si_{n-1}^{{  r_y}} (z)\tf{dy}{\sqrt{1-|y|^2}}\\
=&{   \int_{B_n}r_y^{n-1}\int_{\mathbb S^{n-1}}F(y,r_y\om)d\si_{n-1} (\om)\tf{dy}{\sqrt{1-|y|^2}}    } \\
=&\int_{B_n}r_y^{n-1}\int_{B_{n-1}} {  \big[F(y,r_y \om',r_y\om_{n})  
+ F(y,r_y \om',-r_y\om_{n}) \big]}\tf{d\om'}{\sqrt{1-|\om'|^2}}\tf{dy}{\sqrt{1-|y|^2}}\\
=&\int_{B_n}r_y^{n-1}\int_{r_yB_{n-1}}{  \big[ F(y,z',z_n)+F(y,z',-z_n)\big]}\tf{r_y^{1-n}dz'}{\sqrt{1-|\om'|^2}}\tf{dy}{\sqrt{1-|y|^2}} \\
= & {  \int_{B_{n}}\int_{r_yB_{n-1}} \big[ F(y,z',z_n) +  F(y,z',-z_n) \big]    \tf{dydz'}{\sqrt{1-|y|^2-|z'|^2}}, }
\end{align*}
as one can easily verify that 
$\sqrt{1-|\om'|^2}\sqrt{1-|y|^2}=\sqrt{1-|y|^2-|z'|^2}$. Using that  
$B_{2n-1}$ is equal to the disjoint union of the sets $  \{ (y,r_yv):\,\,   {v\in B_{n-1}} \}$ over all $y\in B_n$,    we see that the last double  integral is equal to the expression in 
\eqref{e04023}, as   claimed.
\epf



\begin{thebibliography}{99}

\bibitem{Bourgain1986}
Bourgain, J.,  
\emph{Averages in the plane over convex curves and maximal operators}.
 Journal d'Analyse Math{\'e}matique  {\bf 47}(1)  (1986), 69--85.





\bibitem{Calderon1979}
Calder{\'o}n, C.~P.,  
\emph{Lacunary spherical means}.
 Illinois Journal of Mathematics   {\bf 23}(3)  (1979), 476--484.

\bibitem{Carbery1985}
Carbery, A.,  
\emph{Radial Fourier multipliers and associated maximal functions}.
 North-Holland Mathematics Studies   {\bf 111}  (1985), 49--56.


\bibitem{Coifman1978a}
Coifman, R.~R. and Weiss, G.,  
 Review: R. E. Edwards and G. Gaudry, {\it Littlewood-Paley and multiplier  theory}.
 Bulletin of the American Mathematical Society   {\bf 84}(2)  (1978), 242--250.

\bibitem{Cowling1979}
Cowling, M. and Mauceri, G., 
\emph{On maximal functions}. 
 Milan Journal of Mathematics   {\bf 49}(1)  (1979), 79--87.

\bibitem{DRdF} Duoandikoetxea, J. and Rubio de Francia, J.-L., 
\emph{Maximal and singular integral operators via Fourier transform estimates}. 
 Inventiones Mathematicae  {\bf 84} (1986), 541--561.

 
\bibitem{Duoandikoetxea1996}
Duoandikoetxea, J. and Vega, L.,  
\emph{ Spherical means and weighted inequalities}.
 Journal of the London Mathematical Society   {\bf 53}(2)  (1996), 343--353.

\bibitem{GGIPS}
Geba, D., Greenleaf, A., Iosevich, A., Palsson, E. and Sawyer, E. 
{\em Restricted convolution inequalities, multilinear operators and applications}. 
Mathematical Research Letters {\bf 20} (2013), no. 4, 675--694.

\bibitem{CFA} Grafakos, L.,  {\it Classical Fourier Analysis}, Third Edition. 
Graduate Texts in Mathematics, {\bf 249}, Springer, New York, 2014.


\bibitem{GraHeHon}
 Grafakos, L., He, D., and  Honz{\'\i}k P., {\em Rough bilinear singular
  integrals}. Submitted (2016).
  
\bibitem{GHH1}
Grafakos, L., He, D., and  Honz{\'\i}k  P., 
{\em The H\"ormander multiplier theorem, II: The bilinear local $L^2$ case}. Submitted (2016).
  
\bibitem{Greenleaf1981}
Greenleaf, A.,  
{\em Principal curvature and harmonic-analysis}.
 Indiana University Mathematics Journal   {\bf 30}(4)  (1981), 519--537.


\bibitem{GIKL}
Greenleaf, A., Iosevich, A., Krause, B., and Liu, A.,
{\em Bilinear generalized Radon transforms in the plane}.
https://arxiv.org/abs/1704.00861


\bibitem{He}
He, D., 
{\em On bilinear maximal Bochner-Riesz operators}. Submitted (2016).


\bibitem{Magyar2002}
Magyar, A., Stein, E., and Wainger, S., 
{\em Discrete analogues in harmonic analysis: spherical averages}.
 Annals of Mathematics (2nd Ser.)  {\bf 155}(1) (2002), 189--208. 

\bibitem{Mockenhaupt1992}
Mockenhaupt, G., Seeger, A., and Sogge, C.~D.,  
{\em Wave front sets, local smoothing and Bourgain's circular maximal theorem}.
 Annals of Mathematics (2nd Ser.)  {\bf 136}(1)   (1992),  207--218. 


\bibitem{RdF86}
Rubio~de Francia, J.~L.,  
{\em Maximal functions and Fourier transforms}.
 Duke Mathematical Journal  {\bf 53}(2) (1986),  395--404.

\bibitem{Schlag1998}
Schlag, W., 
{\em A geometric proof of the circular maximal theorem}.
 Duke Mathematical Journal  {\bf 93}(3) (1998), 505--534.



\bibitem{Stein1976}
Stein, E.~M.,  
{\em Maximal functions: spherical means}.
 Proceedings of the National Academy of Sciences,
  {\bf 73}(7) (1976), 2174--2175.
  
\bibitem{stein-book2} Stein, E. M.,  
\emph{Harmonic Analysis, Real Variable Methods, Orthogonality, and 
Oscillatory Integrals}. 
Princeton Mathematical Series, 43, Monographs in Harmonic Analysis, III, 
Princeton University Press, Princeton, NJ, 1993.  

\end{thebibliography}
\end{document}